\documentclass{amsart}
\usepackage{amssymb}
\usepackage{graphicx}
\newtheorem{theorem}{Theorem}[section]
\newtheorem{lemma}[theorem]{Lemma}
\newtheorem{proposition}[theorem]{Proposition}

\theoremstyle{definition}
\newtheorem{definition}[theorem]{Definition}

\theoremstyle{remark}

\numberwithin{equation}{section}

\newcommand{\divrg}{\mathrm{div}\,}

\begin{document}
\title[Unique continuation]{Strong unique continuation\\for general elliptic equations in 2D}
\author{Giovanni Alessandrini}
\address{Dipartimento di Matematica e Informatica, Universit\`a degli Studi di Trieste, Italy}

\email{alessang@units.it}

\subjclass[2000]{} \keywords{}

\begin{abstract} We prove that solutions to elliptic equations in
two variables in divergence form, possibly non--selfadjoint and
with lower order terms, satisfy the strong unique continuation
property.

\end{abstract}

\maketitle
\section{Introduction}
Given a bounded connected open set $\Omega \subset \mathbb{R}^2$,
we consider weak solutions $u\in W^{1,2}(\Omega)$ to the elliptic
equation
\begin{equation}\label{eqbase}
L u =  0 \ , \ \textrm{in} \ \Omega \ ,
\end{equation}
where $L$ is defined as follows
\begin{equation}\label{opbase}
L u = - \divrg(A\nabla u + u B) + C\cdot \nabla u + du  \   .
\end{equation}
Here, the coefficients $A, B, C, d$ are assumed to satisfy the
hypotheses listed below. For simplicity, and with no loss of
generality, they are assumed to be defined on all of
$\mathbb{R}^2$.

$A=\{a_{ij}\}$ is a positive definite, \emph{possibly
non--symmetric}, $2\times 2$ matrix with
$L^{\infty}(\mathbb{R}^2)$ entries. We express its uniform
ellipticity by the following bounds (see, for instance,
\cite{ANannales} for their equivalence to other customary
formulations of ellipticity). For a given $K\geq 1$
\begin{equation}\label{ellTar}
\begin{array}{ccrllll}
A(x) \xi\cdot \xi&\geq& K^{-1} |\xi|^2&,&\hbox{for every $\xi \in \mathbb R^2$ and for a.e. $x\in\mathbb R^2$ ,}\\
A^{-1}(x) \xi \cdot \xi &\geq & K^{-1} |\xi|^2&,&\hbox{for every
$\xi \in \mathbb R^2$ and for a.e. $x\in\mathbb R^2$ .}
\end{array}
\end{equation}
For a given $q>2$, $B=\{b_{i}\}, C=\{c_{i}\}$ are vector valued
functions belonging to $L^q(\mathbb{R}^2;\mathbb{R}^2)$ and $d$ is
an $L^{q/2}(\mathbb{R}^2)$ function. Altogether, for a given
$\kappa>0$, they are assumed to satisfy
\begin{equation}\label{lots}
\|B\|_{L^q(\mathbb{R}^2)}+\|C\|_{L^q(\mathbb{R}^2)}+\|d\|_{L^{q/2}(\mathbb{R}^2)}
\leq \kappa \ .
\end{equation}
The aim of this note is to prove the following theorem.

\begin{theorem}\label{SUCP}
The operator $L$ has the strong unique continuation property.
\end{theorem}

As is well--known, this means that if a solution $u$ to
\eqref{eqbase} has a zero of infinite order at a point $x_0\in
\Omega$ then $u=0$ identically in $\Omega$.

In order to summarize the past history of results of unique
continuation for elliptic equations in two variables, we may start
with a remark due to Martio \cite{martio}:
\begin{quote}
In the plane uniqueness results can be proved using the theory of
quasiconformal mappings and special representation theorems.
\end{quote}
In fact Martio is referring to Chapter 6 in the book by Bers, John
and Schechter \cite{BJS}, which in turn is mainly based on the
papers by Bers and Nirenberg \cite{bersni} and by Bojarski
\cite{boj}. In these items, a theory of  first order elliptic
systems  in the plane (the so--called Beltrami systems) is
developed, and in particular it is applied to the unique
continuation of second order elliptic equations in
\emph{non--divergence} form, see for instance \cite[\S 6.4]{BJS}.
Thus it seems that up to 1988, a proof of the unique continuation
for equations in divergence form and $L^{\infty}$ coefficients in
the principal part, was not available. On the other hand, assuming
some regularity on the coefficients in the principal part, results
of unique continuation were known, one can refer to Carleman
\cite{carleman33} and Hartman and Wintner \cite{hartmanw}.

In 1992, motivated by  inverse boundary problems \cite{adiaz,
aispow}, the author noticed  that the approach based on Beltrami
equations and on the representation formulas of Bers--Nirenberg
and Bojarski could be used also to prove unique continuation for
pure divergence elliptic equations, that is equations like
\eqref{eqbase}, \eqref{opbase} with $B=C=0$ and $d=0$,
\cite{apreprint}. This result appeared in print soon afterwards in
a joint work with Magnanini \cite{amagnanini}. Next, in
\cite{acourant}, the unique continuation was shown for operators
of the form
\begin{equation}\label{opautagg}
L u = - \divrg(A\nabla u) + du  \   ,
\end{equation}
with $A$ symmetric and $d$ bounded. The idea there was that, on a
sufficiently small disk, we can find a positive solution $w$ to
$Lw=0$. Hence by a classical trick, see for instance Miranda \cite[Ch. I, \S5]{miranda},  we note that any solution $u$ to
$Lu=0$ can be locally factored as $u=wv$ where $v$ solves
\begin{equation}\label{varcost}
 - \divrg(w^2A\nabla v) = 0 \   .
\end{equation}
That is, loosely speaking, $w(- \divrg(A\nabla (w\cdot)) +
dw\cdot) = - \divrg(w^2A\nabla \cdot)$  and we have reduced
ourselves again to a pure divergence form equation. Soon
afterwards, Schulz \cite{schulzzeit} noticed that a slightly more
involved, but similar trick, could be used to treat operators of
the general form \eqref{opbase}, with bounded lower order
coefficients, but still, \emph{symmetric} principal part!

Here we bypass this symmetry obstruction by using a reduction
formula, Proposition \ref{purediv}, on the operator \eqref{opbase}
which involves the use of two positive multipliers $m, w$ instead
of one. These multipliers are obtained as solutions in the small
of two appropriately chosen elliptic equations. In fact, one of
the two multipliers, $w$, is constructed, Definition
\ref{multiplier2}, as a solution of an equation whose coefficients
depend on the previously chosen function $m$, see Definition
\ref{meLtilde}. This forces to take into consideration equations
with unbounded lower order coefficients, and this is the main
reason why, we have assumed since the beginning the integrability
conditions \eqref{lots}.

In the following Section 2 we construct such multipliers and prove
the reduction to a pure divergence equation, Proposition
\ref{purediv}. In the final Section 3 we complete the proof of the
main Theorem \ref{SUCP}.

\section{Preliminaries}
 In what follows, we denote by
$B_R$ a disk of radius $R$ and arbitrary center.
\begin{lemma}\label{meyers}
There exists $p$, $2<p<q$, only depending on $K$ and $q$, and a
number $R_0>0$, only depending on $K, q$ and $\kappa$ such that
for every $R$, $0<R\leq R_0$ and for every $F\in
L^p(B_R;\mathbb{R}^2)$ there exist a unique weak solution $u\in
W^{1,p}_0(B_R)$ to the equation
\begin{equation}\label{local dir}
L u = -\divrg F
\end{equation}
and it satisfies
\begin{equation}\label{localbd}
\|\nabla u\|_{L^p(B_R)} \leq C\|F\|_{L^p(B_R)}
\end{equation}
where $C>0$ only depends on $K$ and $q$.
\end{lemma}
\begin{proof}
This result is a minor, well--known, variation of a celebrated
theorem of Meyers \cite[Theorem 1]{meyers}. In fact, in
\cite{meyers} it is proven that, considering the principal part
$L_0$ of the operator $L$, that is
\begin{equation}\label{lin0}
L_0u = -\divrg ( A\nabla u )  \ ,
\end{equation}
 there exists $p$, $2<p<q$, and $c_0>0$, only
depending on $K$ and $q$, such that for every $R>0$ we have that
$L_0 : W^{1,p}_0(B_R) \mapsto W^{-1,p}(B_R)$ is invertible and the following estimate holds
\begin{equation}\label{coerciv0}
\|L_0 u\|_{W^{-1,p}(B_R)}\geq c_0 \|\nabla
u\|_{L^{p}(B_{R})} \textrm{ for every } u \in W^{1,p}_0(B_R) \ .
\end{equation}
Let us denote by $M= L-L_0$ the remainder first order operator
\begin{equation}\label{resto}
M u= - \divrg(uB) + C\cdot \nabla u + d u  \ .
\end{equation}
Then equation \eqref{local dir} can be rewritten as
\begin{equation}\label{localcontr}
u+ L_0^{-1}M u = L_0^{-1}( -\divrg F ) \ ,
\end{equation}
and the thesis will follow provided we show that, for sufficiently small $R$ the operator  $L_0^{-1}M$ is a contraction on $W^{1,p}_0(B_R)$.
By \eqref{lots} and by a straightforward use of Sobolev
inequalities, one obtains that, for every $R>0$ and for every $u
\in W^{1,p}_0(B_R), v\in W^{1,p^{\prime}}_0(B_R)$, we have
\begin{equation}\label{restoupperbd}
\begin{array}{c}
|<Mu,v>)| = |\int_{B_R}u B\cdot \nabla v + v C\cdot \nabla u + d u v |\leq \\ \\
\leq C\kappa
(R^{2(\frac{1}{2}-\frac{1}{q})}+R^{2(\frac{1}{p}-\frac{1}{q}+\frac{1}{2}-\frac{1}{q})})
\| \nabla v\|_{L^{p^{\prime}}(B_{R})} \|\nabla
u\|_{L^{p}(B_{R_0})} \ ,
\end{array}
\end{equation}
here $<\cdot,\cdot>$ denotes the dual pairing between
$W^{-1,p}(B_R)$ and $ W^{1,p^{\prime}}_0(B_R)$, and the constant
 $C>0$ only depends on $q$ and $p$, that is on $q$ and $K$.
Consequently, by \eqref{coerciv0}, there exists $R_0>0$, only
depending on $q, K$ and $\kappa$ such that  $\|L_0^{-1}M\| \leq \frac{1}{2}$. And \eqref{localbd} follows with $C=\frac{2}{c_0}$.
\end{proof}
\begin{lemma}\label{meyers2}
Under the same assumptions of Lemma \ref{meyers}, and letting
$R_0, p$ as introduced in the same Lemma,  given $F\in
L^q(B_R;\mathbb{R}^2)$ and $f \in L^s(B_R)$ with
$\frac{1}{s}\leq\frac{1}{2}+\frac{1}{q}$, for every $R\leq R_0$
there exist a unique weak solution $u\in W^{1,p}_0(B_R)$ to the
equation
\begin{equation}\label{local dirFf}
L u = -\divrg F + f
\end{equation}
and it satisfies
\begin{equation}\label{smallbd1}
\|\nabla u\|_{L^p(B_R)} \leq
C(R^{2(\frac{1}{p}-\frac{1}{q})}\|F\|_{L^q(B_R)}+R^{2(\frac{1}{p}-\frac{1}{s})+1}\|f\|_{L^s(B_R)}
) \ ,
\end{equation}
and also
\begin{equation}\label{smallbd2}
\| u\|_{L^{\infty}(B_R)} \leq C
R^{1-\frac{2}{p}}(R^{2(\frac{1}{p}-\frac{1}{q})}\|F\|_{L^q(B_R)}+R^{2(\frac{1}{p}-\frac{1}{s})+1}\|f\|_{L^s(B_R)}
) \ ,
\end{equation}
where $C>0$ only depends on $K, q$ and $s$.
\end{lemma}
\begin{proof}
This is also a well-known consequence of Meyers' result
\cite[Theorem 1]{meyers}. In  fact we may easily construct $G\in
L^q(B_R;\mathbb{R}^2)$ such that $- \divrg G = f$ and also
\begin{equation}\label{G}
\|G\|_{L^q(B_R)}\leq C
R^{2(\frac{1}{q}-\frac{1}{s})+1}\|f\|_{L^s(B_R)} \ ,
\end{equation}
where $C>0$ only depends on $q$ and $s$. Hence, applying Lemma
\ref{meyers} with $F$ replaced with $F+G$, we obtain
\eqref{smallbd1}, \eqref{smallbd2} follows by a Sobolev
inequality.
\end{proof}
We record here another result, of the same flavor as the previous
lemmas, which shall be used later on.
\begin{lemma}\label{tipocacciopp}
Let
 $u\in W^{1,2}(\Omega)$  be a weak solution to \eqref{eqbase}
in $\Omega$ , and let $p$, $2<p<q$, be the exponent introduced in
Lemma \ref{meyers}. For any two concentric balls $B_{\rho}\subset
B_r \subset \Omega$, we have
\begin{equation}\label{discacciopp}
\|\nabla u\|_{L^p(B_{\rho})} \leq C r^{2(\frac{1}{p}-1)}
\|u\|_{L^2(B_{r})} \ ,
\end{equation}
where $C>0$ only depends on $K, \kappa, q$ and on the ratio
$\frac{r}{\rho}$.
\end{lemma}
\begin{proof} This is indeed Meyers' higher integrability theorem \cite[Theorem 2]{meyers}. It may also be obtained in a
straightforward manner from Lemma \ref{meyers} above with the aid
of a smooth cutoff function.
\end{proof}
The proposition below provides the main tool in the construction
of the required multipliers.
\begin{proposition}\label{multiplier} Under the same assumptions
as above, and letting $R_0, p$ as before, there exists $R_1$,
$0<R_1\leq R_0$, only depending on $K, \kappa$ and $q$, such that
there exists $u\in W^{1,p}(B_{R_1})$ which is a weak solution to
\begin{equation}\label{eqmulti}
L u =  0 \ , \ \textrm{in} \ B_{R_1} \ ,
\end{equation}
and it satisfies
\begin{equation}\label{multibd1}
\frac{1}{2}\leq u \leq 2 \ ,
\end{equation}
\begin{equation}\label{multibd2}
\|\nabla u\|_{L^p(B_{R_1})}\leq 1 \ .
\end{equation}
\end{proposition}
\begin{proof}
Let $z \in W^{1,p}_0(B_R)$ be the solution to \eqref{local dir}
obtained in Lemma \ref{meyers2} when $F, f$ are replaced with $
-B, -d$, respectively. We may choose $R=R_1$ small enough so that
\begin{equation}\label{multibd1z}
|z| \leq \frac{1}{2} \ ,
\end{equation}
\begin{equation}\label{multibd2z}
\|\nabla z\|_{L^p(B_{R_1})}\leq 1 \ .
\end{equation}
The thesis follows by picking $u=z+1$.
\end{proof}
In the next definitions we construct the multipliers $m, w$ and
two auxiliary elliptic operators.
\begin{definition}\label{meLtilde}
We define the multiplier $m$ as the solution obtained in the
previous Proposition \ref{multiplier}, when, in the operator $L$
introduced in \eqref{opbase} the coefficient vector $B$ is
replaced with $0$. That is, $m$ is a weak solution to
\begin{equation}\label{equazm}
 - \divrg(A\nabla m ) + C\cdot \nabla m + dm = 0 \textrm{ in } B_{R_1}  \
 ,
\end{equation}
and it satisfies the  bounds \eqref{multibd1}, \eqref{multibd2}.
 Consequently, we define in $B_{R_1}$ the following set of
coefficients
\begin{equation}\label{coeftilde}
\begin{array}{l}
\widetilde{A}= m A^{T}\ , \\
\widetilde{B}= m C - A\nabla m \ , \\
\widetilde{C} = m B \ ,
\end{array}
\end{equation}
here the superscript $(\cdot)^T$ denotes the transpose.
Accordingly, we set
\begin{equation}\label{optilde}
\widetilde{L} u = - \divrg(\widetilde{A}\nabla u + u
\widetilde{B}) + \widetilde{C}\cdot \nabla u  \ .
\end{equation}
Observe that the following bounds are easily obtained
\begin{equation}\label{elltilde}
\begin{array}{ccrllll}
\widetilde{A}(x) \xi\cdot \xi&\geq& \frac{1}{2K} |\xi|^2&,&\hbox{for every $\xi \in \mathbb R^2$ and for a.e. $x\in B_{R_1}$ ,}\\
\widetilde{A}^{-1}(x) \xi \cdot \xi &\geq & \frac{1}{2K}
|\xi|^2&,&\hbox{for every $\xi \in \mathbb R^2$ and for a.e. $x\in
B_{R_1}$ ,}
\end{array}
\end{equation}
\begin{equation}\label{lotstilde}
\|\widetilde{B}\|_{L^p(B_{R_1})}+\|\widetilde{C}\|_{L^p(B_{R_1})}
\leq 2(\pi R_1^2)^{\frac{1}{p}-\frac{1}{q}}\kappa +K\ .
\end{equation}
\end{definition}
\begin{definition}\label{multiplier2} Let $R_0, R_1$ and $p$  be as before. By applying Proposition \ref{multiplier}
 to the operator $\widetilde{L}$ we find that there exists $R_2$,
$0<R_2\leq R_1$, and $t$, $2<t< p$ only depending on $K, \kappa$
and $q$, such that there exists $w\in W^{1,t}(B_{R_2})$ which is a
weak solution to
\begin{equation}\label{eqmultitilde}
\widetilde{L} w =  0 \ , \ \textrm{in} \ B_{R_2} \ ,
\end{equation}
and it satisfies
\begin{equation}\label{multibd1w}
\frac{1}{2}\leq w \leq 2 \ ,
\end{equation}
\begin{equation}\label{multibd2w}
\|\nabla w\|_{L^t(B_{R_2})}\leq 1 \ .
\end{equation}
Such a function $w$ shall be our second multiplier.
Let us fix
 any disk $B_{R}\subset \Omega$ with $R\leq R_2$.
In $B_{R}$ we define
\begin{equation}\label{coefhat}
\begin{array}{l}
\widehat{A}= m w A\ , \\
\widehat{B}= w A\nabla m + m w B -m A^{T}\nabla w - m w C \ ,
\end{array}
\end{equation}
and consequently, we set
\begin{equation}\label{ophat}
\widehat{L} u = - \divrg(\widehat{A}\nabla u + u \widehat{B})  \ .
\end{equation}
Note that the following bounds are easily verified.
\begin{equation}\label{ellhat}
\begin{array}{ccrllll}
\widehat{A}(x) \xi\cdot \xi&\geq& \frac{1}{4K} |\xi|^2&,&\hbox{for every $\xi \in \mathbb R^2$ and for a.e. $x\in B_{R}$ ,}\\
\widehat{A}^{-1}(x) \xi \cdot \xi &\geq & \frac{1}{4K}
|\xi|^2&,&\hbox{for every $\xi \in \mathbb R^2$ and for a.e. $x\in
B_{R}$ ,}
\end{array}
\end{equation}
\begin{equation}\label{lotshat}
\|\widehat{B}\|_{L^t(B_{R})} \leq 2K(1+(\pi
R^2)^{\frac{1}{t}-\frac{1}{p}}) +4 (\pi
R^2)^{\frac{1}{t}-\frac{1}{q}} \kappa    \ .
\end{equation}
\end{definition}
\begin{proposition}\label{purediv}
For any $v \in W^{1,2}(B_{R})$ we have
\begin{equation}\label{eqv}
\widehat{L} v =w L( m v )  \ ,
\end{equation}
as elements of $W^{-1,2}(B_{R})$.
\end{proposition}
\begin{proof}
Let $u=mv$ and note that $u \in W^{1,2}(B_{R})$. Let $F\in
L^2(B_R;\mathbb{R}^2)$ be such that $ -\divrg F = Lu $ in the
sense of $W^{-1,2}(B_{R})$. Let $\psi \in C_0^{\infty}(B_{R})$ be
an arbitrary test function. Denote $\varphi = w \psi$ and observe
that $\varphi \in W_0^{1,2}(B_{R})$. Consequently
\begin{equation}
\int A\nabla u \cdot \nabla \varphi + u B\cdot \nabla \varphi +
\varphi C\cdot \nabla u + d u \varphi = \int F\cdot \nabla \varphi
\ ,
\end{equation}
here, and in the rest of this proof, integrals are intended over
$B_{R}$. Using the chain rule on the products $u=m v$, $\varphi =
w \psi$, we obtain
\begin{equation}
\begin{array}{c}
\int (mwA\nabla v \cdot \nabla \psi + m\psi A\nabla v \cdot \nabla
w + \\ \\+ v w A\nabla m \cdot \nabla \psi +  v\psi A\nabla m
\cdot \nabla w +\\ \\+ m v w B\cdot \nabla \psi + m v \psi B\cdot
\nabla w
+\\ \\
 + m w \psi  C\cdot \nabla v + v w \psi   C\cdot \nabla m + \\ \\ + d m v w
 \psi)
=  \int F\cdot \nabla (w \psi) \ .
\end{array}
\end{equation}
By the identity $\psi \nabla v = \nabla(\psi v) - v \nabla \psi$,
we may use the following substitutions
\begin{equation}
m\psi A\nabla v \cdot \nabla w = m A \nabla(\psi v) \cdot \nabla w
- mv A\nabla \psi \cdot \nabla w \ ,
\end{equation}
\begin{equation}
m w\psi C\cdot \nabla v = m w C\cdot \nabla (\psi v) - m w v
C\cdot \nabla \psi \ .
\end{equation}
Therefore
\begin{equation}
\begin{array}{c}
\int (mwA\nabla v \cdot \nabla \psi + v w A\nabla m \cdot \nabla
\psi + m v w B\cdot \nabla \psi - mv A\nabla \psi \cdot \nabla w -
m w v C\cdot \nabla \psi)+\\ \\
+\int (m A \nabla(\psi v) \cdot \nabla w + m w C\cdot \nabla (\psi
v))+\\ \\
+\int ( v\psi A\nabla m \cdot \nabla w + m v \psi B\cdot \nabla w

 +  v w \psi   C\cdot \nabla m  + d m v w
 \psi)
=  \int F\cdot \nabla (w \psi) \ .
\end{array}
\end{equation}
Again by the chain rule, we may substitute
\begin{equation}
v\psi A\nabla m \cdot \nabla w =  A \nabla m \cdot \nabla (v\psi
w) - w A\nabla m \cdot \nabla (v \psi) \ ,
\end{equation}
and obtain
\begin{equation}
\begin{array}{c}
\int (mwA\nabla v \cdot \nabla \psi + v w A\nabla m \cdot \nabla
\psi + m v w B\cdot \nabla \psi - mv A\nabla \psi \cdot \nabla w -
m w v C\cdot \nabla \psi)+\\ \\
+\int (m A \nabla(\psi v) \cdot \nabla w + m w C\cdot \nabla (\psi
v))+\\ \\
+\int (m v \psi B\cdot \nabla w - w A\nabla m \cdot \nabla (v
\psi) ) +\\ \\
+\int(A \nabla m \cdot \nabla (v\psi w)+  v w \psi   C\cdot \nabla
m  + d m v w
 \psi)
=  \int F\cdot \nabla (w \psi) \ .
\end{array}
\end{equation}
Now we note that $v \psi w \in W_0^{1,2}(B_{R})$, therefore by
\eqref{equazm} the fourth integral on the left hand side vanishes.
Then we can rearrange the terms as follows
\begin{equation}
\begin{array}{c}
\int (mwA\nabla v \cdot \nabla \psi + v w A\nabla m \cdot \nabla
\psi + m v w B\cdot \nabla \psi - mv A\nabla \psi \cdot \nabla w -
m w v C\cdot \nabla \psi)+\\ \\
+\int (m A^T \nabla w \cdot \nabla (\psi v) + m w C\cdot \nabla
(\psi v) +m v \psi B\cdot \nabla w - w A\nabla m \cdot \nabla (v
\psi) ) =\\\\= \int F\cdot \nabla (w \psi) \ .
\end{array}
\end{equation}
Again, we note that $v \psi \in W_0^{1,2}(B_{R})$ and by
\eqref{eqmultitilde}, the second integral on the left hand side is
also vanishing. Finally, recalling the notation introduced in
Definition \ref{multiplier2}, we arrive at
\begin{equation}
\begin{array}{c}
\int (\widehat{A}\nabla v \cdot \nabla \psi + v\widehat{B}\cdot
\nabla \psi ) = \int F\cdot \nabla (w \psi) \ ,
\end{array}
\end{equation}
 for every $\psi \in C_0^{\infty}(B_{R})$, ad hence by density, for every  $\psi \in W_0^{1,2}(B_{R})$.
Note, in conclusion,  that the functional $- w \divrg F$ given by
$<- w \divrg F, \psi>= \int F\cdot \nabla (w \psi)$ does indeed
belong to $W^{-1,2}(B_{R})$.
\end{proof}

\section{Proof of the main Theorem}
From now on, let $u$
be a weak solution to \eqref{eqbase}, and let us fix
 any disk $B_{R}\subset \Omega$ with $R< R_2$. We denote
\begin{equation}\label{vdef}
v = \frac{u}{m} \ ,
\end{equation}
where $m$ is the function introduced in Definition \ref{meLtilde}.
Note that, by Proposition \ref{multiplier} and by Lemma
\ref{tipocacciopp}, $v \in W^{1,t}(B_{R})$ and by Proposition
\ref{purediv}
\begin{equation}\label{eqv0}
\widehat{L} v =0  \ ,
\end{equation}
in the weak sense.

The advantage is that from a pure divergence elliptic equation we
can easily pass to a first order elliptic system of Beltrami type.
The procedure is well-known \cite{apreprint,
amagnanini,schulzzeit, ANannales}. Denote
\begin{equation}\label{J}
J=\left(
\begin{array}{ccc}
0&-1\\
1&0
\end{array}
\right)\ ,
\end{equation}
then, by \eqref{eqv0}, $J(\widehat A \nabla v + v\widehat B)$ is
weakly curl-free  in $B_R$ and therefore there exists a function $\tilde v \in W^{1,t}(B_R)$, unique up
to an additive constant, such
that
\begin{equation}\label{firstord}
 \nabla \tilde v= J (\widehat A \nabla v + v\widehat B)  \ ,
\end{equation}
and, since $t>2$, $\tilde v$ is also H\"{o}lder continuous, thus
we can normalize it by setting $\tilde v(x_0) = 0$, where $x_0$
denotes the center of $B_R$. Setting
\begin{equation}\label{deff}
 f=v+i \tilde v
\end{equation}
one has $f \in W^{1,t}(B_R;\mathbb C)$ and, according to Bers and
Nirenberg \cite{bersni}, one can rewrite \eqref{firstord}, in
terms of the complex coordinate $z=x_1+i x_2$, as follows
\begin{equation}\label{beltrami}
\begin{array}{ll}
f_{\bar{z}}=\mu f_z +\nu \overline{f_z} + \alpha f + \beta
\overline f \ , & \hbox{in $B_R$}\ ,
\end{array}
\end{equation}
where, the so called complex dilatations $\mu , \nu$ only depend
(and can be explicitly expressed \cite{ANannales}) on $\widehat
A$, and the lower order coefficients $\alpha, \beta$ only depend
on $\widehat A, \widehat B$. Moreover the following bounds are
easily proven
\begin{equation}\label{munubd}
|\mu|+|\nu|\leq k < 1 \ ,\textrm{ a.e. in } B_R \ ,
\end{equation}
where, in view of the  ellipticity condition \eqref{ellhat}, the
constant $k$ only depends on $K$, see \cite[Proposition
1.8]{ANannales} for a sharp bound. For the lower order
coefficients, recalling \eqref{lotshat}, one can obtain
\begin{equation}\label{lotsbelt}
\|\alpha\|_{L^t(B_{R})}+ \|\beta\|_{L^t(B_{R})} \leq C  \ .
\end{equation}
where $C>0$ only depends on $K, \kappa$ and $q$.

We can now invoke the well-known representation theorem for
solutions of equations of the form \eqref{beltrami}.
\begin{theorem}\label{BersNir}
There exist a $k$-quasiconformal mapping $\chi$ from $\mathbb{C}$
onto itself, a holomorphic function $F$ on $\chi (B_R)$ and a
complex valued H\"{o}lder continuous  function $s$ on $B_R$ such
that
\begin{equation}\label{represform}
f= e^{s}F(\chi).
\end{equation}
Moreover we have that the function $\chi$ and its inverse
$\chi^{-1}$ satisfy the following H\"older continuity properties
\begin{equation}\label{holderchi}
|\chi(z)-\chi(\zeta)|\leq C|z-\zeta|^{\eta},\quad\text{for any
}z,\zeta\in B_R \ ,
\end{equation}
\begin{equation}\label{holderchiinv}
|\chi^{-1}(z)-\chi^{-1}(\zeta)|\leq C|z-\zeta|^{\eta},\quad
\text{for any }z,\zeta\in \chi (B_R) \ ,
\end{equation}
and
\begin{equation}\label{holderchi}
|s(z)-s(\zeta)|\leq C|z-\zeta|^{\eta},\quad\text{for any
}z,\zeta\in B_R \ ,
\end{equation}
where $C$ and $\eta$, $0<\eta<1$, only depend on $K, \kappa$ and
$q$ .
\end{theorem}
\begin{proof} This is a celebrated theorem of Bers and Nirenberg \cite[page 116]{bersni}, see also Bojarski \cite[Theorem 4.3]{boj}
and the book  \cite[Section 6.3]{BJS}. \end{proof}

It is now evident that if $f$ is nontrivial then it may vanish
only up to finite order, in fact in \eqref{represform} the
exponential $e^s$ never vanishes, and $F(\chi)$ may have only
isolated zeroes of finite order in view of \eqref{holderchiinv}.
Only one small step remains in order to show the strong unique
continuation property for $u = m \Re (e^{s}F(\chi))$.
\begin{lemma}\label{boundvtilde}
Let $u\in W^{1,2}(\Omega)$  be a weak solution to \eqref{eqbase}
in $\Omega$ , and let $p$, $2<p<q$, be the exponent introduced in
Lemma \ref{meyers}. Let $v, \tilde v$ as introduced above. For any
two balls $B_{\rho}, B_r$ concentric to $B_R$, $\rho < r < R$, we
have
\begin{equation}\label{Linf}
\|\tilde v\|_{L^{\infty}(B_{\rho})} \leq C
\|v\|_{L^{\infty}(B_{r})} \ ,
\end{equation}
where $C>0$ only depends on $K, \kappa, q$ and on the ratio
$\frac{r}{\rho}$.
\end{lemma}
\begin{proof}
These bounds are straightforward consequences of \eqref{firstord},
by the use of Lemma \ref{tipocacciopp} applied to the operator
$\widehat L$ and by Sobolev inequalities. Note that use is made of
the normalization $\tilde v(x_0)=0$, where $x_0$ is the center of
$B_R$.
\end{proof}
\begin{proof}[Proof of Theorem \ref{SUCP}]
Assume that a solution $u$ to $Lu=0$ has a zero of infinite order
at a point $x_0 \in \Omega$, let $R< R_2$ such that the disk
$B_R$, centered at $x_0$, is contained in $\Omega$. By
\eqref{vdef} also $v$ has a zero of infinite order at  $x_0$, and
by \eqref{Linf}, the same occurs to $\tilde v$. Hence also $f$,
given by \eqref{deff}, does the same. By Theorem \ref{BersNir} we
obtain that $f$, and hence $u$ are identically zero in $B_R$. Then
a standard continuity argument yields that  $u$ is identically
zero in $\Omega$.
\end{proof}

\textbf{Concluding Remark.} In previous studies, \cite{aesca} by
Escauriaza and the author and \cite{Arrv} by Rondi, Rosset,
Vessella and the author, it has been ascertained that, when lower
order terms are absent, or when the operator $L$ is in the
self--adjoint form $Lu = - \divrg(A\nabla u) + du $, with $A$
symmetric, the representation Theorem \ref{BersNir} enables also
to obtain quantitative estimates of unique continuation, such as
doubling inequalities \cite[Proposition 2]{aesca}, three--spheres
inequalities \cite[Proposition 1]{aesca} and \cite[Theorem
1.10]{Arrv}, estimates of propagation of smallness \cite[Theorems
5.1, 5.3]{Arrv} and stability estimates for Cauchy problems
\cite[Theorems 1.9, 7.1]{Arrv}.  In view of the reduction to pure
divergence form obtained in Proposition \ref{purediv}, all such
types of results can be extended to equations of the form
\eqref{eqbase}, \eqref{opbase} treated here. We refrain from
details for the sake of brevity.

\providecommand{\bysame}{\leavevmode\hbox
to3em{\hrulefill}\thinspace}
\providecommand{\MR}{\relax\ifhmode\unskip\space\fi MR }
\providecommand{\MRhref}[2]{%
  \href{http://www.ams.org/mathscinet-getitem?mr=#1}{#2}
} \providecommand{\href}[2]{#2}

\end{document}